\documentclass{amsart}
\usepackage{amssymb,amsmath,amsthm,amsfonts, amsaddr}
\usepackage{mathtools}

\usepackage{url}

\usepackage{enumitem}
\setlist[enumerate]{}

\usepackage{xcolor} 
\usepackage{tikz} 
\usepackage{tikz-cd}

\newtheorem{thm}[subsection]{Theorem}
\newtheorem{defn}[subsection]{Definition}

\newtheorem{cor}[subsection]{Corollary}

\theoremstyle{definition}

\newtheorem{rmk}[subsection]{Remark}

\begin{document}

\begin{abstract}
 We show that an idempotent lies in the center if it commutes with the other idempotents in the ring.  Furthermore, we introduce a partition of the set of idempotents and show that the automorphisms of the ring act transitively on each equivalence class. (Mathematics Subject Classification : 05C25, 16B99)
 \end{abstract}

\title{Centrality and Partition of Idempotents}

\author{Vineeth Chintala}

\address{\small{Indian Institute of Science, Bengaluru, India.}}

\email{vineethreddy90@gmail.com}

\maketitle

\section{Introduction}

Idempotents play a critical role in the structure of rings.  We first show that an idempotent lies in the center if it commutes with the other idempotents.  Furthermore, we partition the set of idempotents into equivalence classes and show that  
\begin{enumerate}[label=\roman*.]
 \item If $e_1, e_2$ are two idempotents in the same equivalence class, then there is an automorphism taking $e_1$ to $e_2$. (Theorem~\ref{thm5}).
 \item For finite rings over $\mathbb{F}_q$, we prove that the number of idempotents in an equivalence class is $q^k$ for some $k$ (Theorem ~\ref{count}).
    \end{enumerate}

\subsection{Notation}
$R$ can be any associative ring with identity. An element $x$ is called an idempotent if $x^2=x$.
An element $c$ is said to be in the center of the ring if $cx =xc$ for all elements $x \in R$.

%


\section{ Lying in the center}\label{id}

The results in this section are presumably known, though probably not in the version stated here.
The importance of the following theorem will become clear when we introduce a partition of a graph of idempotents. 

\begin{thm}\label{thm0}
Let $e$ be an idempotent in $R$. If $e$ does not lie in the center, then $R$ contains another idempotent $e'$ such that  
\[ (ee' =e, {\hskip 2mm} e'e =e')
\text{\hskip 5mm  or \hskip 5mm }
(ee' =e', {\hskip 2mm} e'e=e).\]
\end{thm}
\begin{proof}
Since $e$ does not lie in the center, there is an element $b \in R$ such that 
$be - ebe \neq eb - ebe.$
Let $u = be - ebe$ and $v = eb - ebe$. Then $u^2 = v^2 = 0$.
\vskip 2mm
\noindent The idempotents $e_l = e +u$ and $e_r = e + v$ satisfy the required properties 
\[ee_r =e, \hskip 2mm e_re =e_r \hskip 5mm | \hskip 5mm ee_l =e_l, \hskip 2mm e_le=e.\]
Clearly $e_r, e_l$ are distinct from $e$ whenever $u, v$ are non-zero.
Since $u \neq v$, both $u,v$ can't be simultaneously zero.   
\end{proof}
\vskip 2mm

\vskip 2mm
\begin{cor}\label{cor0}
Let $e$ be an idempotent in $R$.  If $e$ commutes with all the idempotents of $R$, then $e$ lies in the center.
\end{cor}


%


\section{Equivalence classes of Idempotents}\label{idpart}

 Let $\gamma(R)$ be a directed graph, where the vertices are idempotents of $R$, and there is an edge $e \longrightarrow e'$ if and only if $ee' =0$. Then $e'$ is said to be an out-neighbour of $e$ and $e$ is said to be an in-neighbour of $e'$.  
 
 \begin{rmk}
 Zero-divisor graphs (for commutative rings) were introduced by I. Beck in \cite{B} and later extended to the noncommutative setting (see \cite{AM}, \cite{R}). Departing from the literature, we focus here on the idempotents (instead of all the zero-divisors) and prove results that hold for arbitrary noncommutative rings.
\end{rmk}

\begin{defn}
\hfill 
\begin{itemize}
\item Let $e \sim_o e'$ if the idempotents $e, e'$ have the same out-neighbours in $\gamma(R)$. Clearly $\sim_o$ is an equivalence relation. 

 \item Let $\mathcal{O}_e := \{x \in R : ex =e ,xe = x\}$.
 \end{itemize}
\end{defn}

\noindent Observe that all elements of $\mathcal{O}_e$ are idempotents. Indeed if $ a \in \mathcal{O}_e$, then $a^2 = (ae)a = a(ea) =ae =a$.

\begin{thm}\label{thm4}
$e' \in \mathcal{O}_e$ if and only if $e' \sim_o e $.
\end{thm}
\begin{proof}
Let $e' \sim_o e$. Then $e'(1-e) =0$. Therefore $e' = e'e$. Similarly $e(1-e') =0$ and so $e = ee'$.
For the converse, suppose $er =0$. Then $e'r = (e'e)r = 0$.
\end{proof}

Similarly one can define an equivalence relation $\sim_i$ corresponding to the in-neighbours in $\gamma(R)$.

\begin{tabular}{lcl}
\multicolumn{3}{c}{} \\
\hline
 		   & Graph-theoretic & Algebraic \\
		   \hline
\\
 $e' \sim_o e$        & same out-neighbours               		&  $ee' = e$ and $e'e =e'$  \\
 $e' \sim_i e$        & same in-neighbours                  		& $ee' = e'$ and $e'e =e$   \\\\
\end{tabular}

\begin{defn}
Let $\mathcal{I}_e = \{x \in R : ex =x ,xe = e\}$. 
\end{defn}

It follows from the definitions that $\mathcal{I}_e \cap \mathcal{O}_e = \{e\}$. In other words, no two idempotents have the same (in and out) neighbours in $\gamma(R)$.

\subsection{Automorphisms of $\gamma(R)$}
 Any automorphism of the ring translates to an automorphism of the graph. 
\vskip 2mm

\begin{thm}\label{thm5}
Let $e_{1} \sim_o e_2 $. Then there is an inner automorphism of $R$ which sends $e_1$ to $e_2$.  
\end{thm}
\begin{proof}
Take $u = e_2-e_1$. Then $u^2 =0$ and $(1+u)^{-1} = 1-u$. 
Consider the inner automorphism  of $R$ where \[x \rightarrow (1+u)x(1-u).\]
This automorphism takes $e_1$ to $e_2$.
\end{proof}

Therefore there are at least $|\mathcal{O}_e|$ distinct automorphisms stabilizing $\mathcal{O}_e$. Similarly if $e_{1} \sim_i e_2 $, then there is an automorphism which sends $e_1$ to $e_2$.

\vskip 2mm
 Theorem $\ref{thm0}$ tells us that at least one of the sets $\mathcal{I}_e, \mathcal{O}_e$ is nontrivial whenever $e$ is not in the center. In fact we can use the equivalence relations to say more.

\begin{thm} \label{count}
Let $R$ be a finite algebra over $\mathbb{F}_q $. Suppose $e$ is a non-central idempotent, and $\mathcal{X}_e \in \{\mathcal{I}_e, \mathcal{O}_e \}$. Then $|\mathcal{X}_e| = q^k$ for some integer $k$. 
\end{thm}
\begin{proof}
Let $V$ be the vector space over $\mathbb{F}_q$ generated by the elements of $\mathcal{X}_e$. Let $v = \sum\limits_{i=1}^n c_ie_i$, where $e_i \in \mathcal{X}_e$ and $c_i  \in \mathbb{F}_q $.  Then \[v^2 =\big( \small {\sum\limits_{i=1}^n }c_i \big)v.\]
We can choose a basis of $V$ consisting of idempotents $\{e_1, \cdots, e_n \} \subseteq \mathcal{X}_e$.  Any element of $\mathcal{X}_e$ is of the form $e = \sum\limits_{i=1}^n c_ie_i$, where $\sum\limits_{i=1}^n c_i = 1$. There are $q^{n-1}$ such elements.
\end{proof}
\vskip 5mm

%
%


%


\vskip 2mm

\textbf{Acknowledgements.}  The author was supported by the DST-Inspire faculty fellowship in India.


\vskip 2mm 


\begin{thebibliography}{FKS04}

\vskip 2mm 
\bibitem{AM} Akbari, S.,  Mohammadian, A. (2006) Zero-divisor graphs of non-commutative rings, \textit{Journal of Algebra}, 296 , pp. 462--479. 

\bibitem{B} Beck, I. (1988) Coloring of commutative rings, \textit{Journal of Algebra}, 116 (1) : 208–-226. 

\url{https://doi.org/10.1016/0021-8693(88)90202-5}

\bibitem{L} Lam, T. Y., A first course in noncommutative rings,
GTM 131, Springer-Verlag, 1991, DOI : 10.1007/978-1-4419-8616-0.

\bibitem{R} Redmond, S. (2002)
The zero-divisor graph of a noncommutative ring
\textit{Int. J. Commut. Rings}, 1, pp. 203--211.

\end{thebibliography}
\end{document}